\documentclass[a4paper]{article}
\usepackage{a4}
\usepackage{amssymb}
\usepackage{amsthm}
\usepackage{hyperref}
\usepackage{epic,eepicemu}

\newcommand{\Z} {\mathbb{Z}}
\newcommand{\N} {\Z_{>0}}
\newcommand{\Nn}{\Z_{\ge 0}}
\newcommand{\Q} {\mathbb{Q}}
\newcommand{\F} {\mathbb{F}}
\newcommand{\Fp}{ {\mathbb{F}_p} }
\newcommand{\Y} {\mathcal{Y}}
\newcommand{\C} {\mathcal{C}}
\newcommand{\Cp}[1]{\mathcal{C}^#1_p}

\makeatletter
\newcommand{\mod}{\ensuremath{\mathop{\operator@font {mod}}\nolimits}}
\newcommand{\res}{\ensuremath{\mathop{\operator@font {res}}\nolimits}}
\newcommand{\rad}{\ensuremath{\mathop{\operator@font {rad}}\nolimits}}
\newcommand{\rk} {\ensuremath{\mathop{\operator@font {rk}} \nolimits}}
\newcommand{\im} {\ensuremath{\mathop{\operator@font {im}} \nolimits}}
\newcommand{\col}{\ensuremath{\mathop{\operator@font {col}}\nolimits}}
\makeatother

\newcommand{\Md} {\big|}

\newcommand{\lra}{\longrightarrow}
\newcommand{\slra}[1]{\stackrel{#1}{\lra}}
\newcommand{\cp} {\subset_p}
\newcommand{\cmp}[1]{\equiv #1 \, (\mod p)}
\newcommand{\cmps}[1]{\equiv #1 \, (\mod p^2)}

\newcommand{\Sig}[1]{\Sigma_{#1}}
\newcommand{\la} {\lambda}
\newcommand{\Slp}{S^\la_{\Fp}}
\newcommand{\Smp}{S^\mu_{\Fp}}
\newcommand{\Slz}{S^\la_\Z}
\newcommand{\Smz}{S^\mu_\Z}
\newcommand{\Slq}{S^\la_\Q}
\newcommand{\SlR}{S^\la_R}
\newcommand{\Sli} {\SlR\!\uparrow^{\Sig{n+1}}_{\Sig n}}
\newcommand{\Slzi}{\Slz\!\uparrow^{\Sig{n+1}}_{\Sig n}}
\newcommand{\Slr} {\SlR\!\downarrow^{\Sig n}_{\Sig{n-1}}}
\newcommand{\Slzr}{\Slz\!\downarrow^{\Sig n}_{\Sig{n-1}}}
\newcommand{\Hlz}[1]{H^#1(\Sig n,\Slz)}
\newcommand{\Hlp}[1]{H^#1(\Sig n,\Slp)}
\newcommand{\HlR}[1]{H^#1(\Sig n,\SlR)}
\newcommand{\pad}{\sum_{i=1}^{r}c_ip^i}

\newtheorem{Theorem}{Theorem}[section]
\newtheorem{Lemma}[Theorem]{Lemma}
\newtheorem{Corollary}[Theorem]{Corollary}

\newtheorem{Proposition}[Theorem]{Proposition}
\newtheorem{Conjecture}[Theorem]{Conjecture}

\makeatletter 
\let\my@makefnmark\@makefnmark
\renewcommand{\@makefnmark}{}
\makeatother

\begin{document}

\title{Low-degree Cohomology\\of Integral Specht Modules}
\author{Christian Weber}
\maketitle

\begin{abstract}

We introduce a way of describing cohomology of the symmetric groups $\Sig n$ with coefficients in Specht modules. We study $\HlR i$ for $i \in \{0,1,2\}$ and $R = \Z$, $\Fp$. The focus lies on the isomorphism type of $\Hlz 2$. Unfortunately, only in few cases can we determine this exactly. In many cases we obtain only some information about the prime divisors of $|\Hlz 2|$. The most important tools we use are the Zassenhaus algorithm, the Branching Rules, Bockstein type homomorphisms, and the results from \cite{BKM}.
\footnotetext{ \hspace*{-4mm}


Published in Experimental Mathematics 18 (2009), no. 1, 85-95.
}

\end{abstract}


\section{Introduction}

The aim of this work is to determine low-degree cohomology of symmetric groups with coefficients in Specht modules $\SlR$, where $R = \Z$ or $\Fp$. We pay special attention to $\Hlz 2$.

In \cite{KP} and \cite[5. Some applications]{BKM} we find most of the known facts about cohomology of certain irreducible $\Fp\Sig n$-modules. From that we obtain some information about cohomology of certain Specht modules with the help of long exact cohomology sequences (cf. Section \ref{comp}).

For the case $R = \Z$ even less is known. But it is advisable to study integral and $\Fp$-cohomology simultaneously, because there is an interesting connection, which can be described as follows.
\\ 

Let $p$ be a prime. For $R \in \{\Z, \Fp\}$ and $i \ge 0$ let
$$
\C^i(R) := \bigcup_{n \in \Nn} \{ \la \vdash n \mid \HlR i \ne 0 \}.
$$
Further, for $i \ge 1$ we consider
$$
\Cp i := \bigcup_{n \in \Nn} \{ \la \vdash n \mid p \mbox{ divides } |\Hlz{i}| \} \subseteq \C^i(\Z).
$$
Then we obtain the following:

\begin{Lemma} \label{graph connection}
\begin{enumerate}
\item For $i \ge 1$ we have $\C^i(\Fp) = \Cp i \cup \Cp{{i+1}}$.
\item $\C^0(\Fp) = \C^0(\Z) \dot{\cup} \Cp 1$.
\end{enumerate}
\end{Lemma}

The first statement has a more general background, which is treated in Section \nolinebreak \ref{bock}. The lemma follows from Lemmas \ref{bockstein general} and \ref{bockstein H0}.

With the help of Lemma \ref{graph connection} (b), Lemma \ref{bockstein H0} and \cite[24.4 Theorem]{J}, we can describe $\Cp 1$ completely:

\begin{Corollary}
We have $\la := (\la_1,\ldots,\la_{l(\la)}) \in \Cp 1$ if and only if $\la \ne (n)$ and $\la_i \equiv -1 (\mod p^{z_i})$ for all $1 \le i \le l(\la)$, where $z_i := \min\{r \in \Nn \mid p^r > \la_{i+1}\}$.
\end{Corollary}

Let $\Y := \{ \la \vdash n \mid n \in \Nn \}$ be the set of all partitions of all nonnegative integers. The elements of $\Y$ form the vertices of the partition graph, where an arrow points from $\la$ to every element of $\la+$. (As usual, $\la+$ denotes the set of partitions obtained from $\la$ by adding a node, and $\la-$ denotes the set of partitions obtained from $\la$ by removing a node.)

By considering the induced subgraph, we obtain a graph structure for every subset $S \subseteq \Y$. If an arrow points from $\la$ to $\mu$ for $\la, \mu \in S$, we call $\la$ a {\em predecessor of $\mu$ in $S$} and $\mu$ a {\em successor of $\la$ in $S$}.

\begin{Theorem} \label{Cp structure}
Let $\la \in \Cp i$.
\begin{enumerate}
\item There exists a successor of $\la$ in $\Cp i$.
\item If $\la \vdash n$ with $p \nmid n$, then there exists a predecessor of $\la$ in $\Cp i$.
\end{enumerate}
\end{Theorem}

This will be proved in Section \ref{branch}. The analogous statement for $\C^i(\Fp)$ follows from Theorem \ref{Cp structure} by Lemma \ref{graph connection} (a):

\begin{Corollary} \label{CFp structure}
Let $\la \in \C^i(\Fp)$.
\begin{enumerate}
\item There exists a successor of $\la$ in $\C^i(\Fp)$.
\item If $\la \vdash n$ with $p \nmid n$, then there exists a predecessor of $\la$ in $\C^i(\Fp)$.
\end{enumerate}
\end{Corollary}

Our aim is to determine $\Cp 2$. The previous statements suggest the following strategy: First we try to find some (or better, all) partitions in $\Cp 2$ without predecessor, and then we try to find some (or better, all) successors of every known partition in $\Cp 2$. This means that we search for paths in the graph $\Cp 2$. In this context, the term {\em path} means an infinite simple path, whose start vertex has no predecessor in $\Cp 2$.

Before going into that, we have to prepare the tools: In Section \ref{comp} we collect all information we can get about $\C^1(\Fp)$ and $\C^2(\Fp)$. In Section \ref{zas} we will take a look at a computational approach to $\Hlz 2$ based on the Zassenhaus algorithm. The computations were executed in GAP. In a few cases the algorithm even provides theoretical results.

Finally, in Section \ref{path} we combine the previous results in order to obtain some first statements about certain paths in $\Cp 2$. Further, we formulate some conjectures based on the computational results.

Section 3, most of the results in Section 5, and the conjectures in Section 6 are based on the author's diploma thesis \cite{web}, which was prepared under the supervision of Prof. Gerhard Hi\ss.

\section{Bockstein type homomorphisms} \label{bock}

In the following, let $G$ be a finite group and $M$ a $\Z G$-module. We consider the cohomology groups $H^i(G,M)$ of $G$ with coefficients in $M$ for $i \in \Nn$.

An important tool in cohomology theory is exact sequences. For every short exact sequence of $\Z G$-modules there exist natural homomorphisms, which build up a long exact cohomology sequence. We take a look at the following example:
$$
0 \lra \Z^k \stackrel{\cdot p}{\lra} \Z^k \lra \Z^k/p\Z^k \lra 0.
$$
The corresponding long exact cohomology sequence looks like this:
$$
\begin{array}{ccccccccc}
0 & \lra & H^0(G,\Z^k) & \slra{\pi_0} & H^0(G,\Z^k) & \slra{\delta_0} & H^0(G,\Z^k/p\Z^k) \\
& \slra{\beta_0} & H^1(G,\Z^k) & \slra{\pi_1} & \dots & \slra{\delta_{i-1}} & H^{i-1}(G,\Z^k/p\Z^k) \\
& \slra{\beta_{i-1}} & H^i(G,\Z^k) & \slra{\pi_i} & H^i(G,\Z^k) & \slra{\delta_i} & H^i(G,\Z^k/p\Z^k) \\
& \slra{\beta_i} & H^{i+1}(G,\Z^k) & \slra{\pi_{i+1}} & H^{i+1}(G,\Z^k) & \slra{\delta_{i+1}} & \ldots
\end{array}
$$
The homomorphisms $\pi_i$ are defined by
$$
\pi_i   : H^i(G,\Z^k) \lra H^i(G,\Z^k), x \mapsto px.
$$
In the case of the trivial module $\Z^1$, the connecting homomorphisms $\beta_i$ are known as {\em Bockstein homomorphisms}. We are going to use this term for arbitrary $\Z G$-modules $\Z^k$.
\\ 

By \cite[Corollary 6.5.10]{We}, $H^i(G,\Z^k)$ is a finite abelian group for $i \ge 1$. Hence, by the fundamental theorem for finitely generated abelian groups, its isomorphism type is characterized by its {\em elementary divisors} $e_1, \ldots, e_r \ge 1$, where $e_j \Md e_{j+1}$ for $1 \le j \le r-1$ and
$$
H^i(G,\Z^k) \cong \bigoplus_{j=1}^{r} \Z/e_j\Z.
$$

For $i \ge 1$ let $x_i$ be the $p$-rank of $H^i(G,\Z^k)$, that is, the number of direct summands of the maximal elementary abelian $p$-subgroup. (Since $H^i(G,\Z^k)$ is a finite abelian group, $x_i$ equals the number of elementary divisors divisible by $p$.) Further, let $d_i$ be the dimension of $H^i(G,\Z^k/p\Z^k)$ for $i \ge 0$. Then we obtain the following statement:

\begin{Lemma} \label{bockstein general}
For $i \ge 1$ we have $d_i = x_i + x_{i+1}$.
\end{Lemma}
\begin{proof}
Since $\pi_i$ is multiplication by $p$, we have for all $i \ge 0$
$$
\im(\pi_i) \cong \bigoplus_{j=1}^{r-x_i} \Z/e_j\Z \oplus  \bigoplus_{j=r-x_i+1}^{r} \Z/{\textstyle \frac{e_j}{p}}\Z,
$$
$$
\ker(\pi_i) \cong \bigoplus_{j=1}^{x_i} \Z/p\Z.
$$
Further, we know that
$$
\ker(\pi_{i+1}) = \im(\beta_i) \cong H^i(G,\Z^k/p\Z^k)/\ker(\beta_i),
$$
$$
\ker(\beta_i) = \im(\delta_i) \cong H^i(G,\Z^k)/\ker(\delta_i) = H^i(G,\Z^k)/\im(\pi_i)  \cong \ker(\pi_i).
$$
By considering the group orders, we obtain
$$
x_{i+1} = d_i-\dim(\ker (\beta_i)) = d_i-x_i.
$$
Now the claim follows.
\end{proof}

\begin{Corollary} \label{bockstein cor}
Let $i \ge 1$.
\begin{enumerate}
\item If $p \Md |H^i(G,\Z^k)|$, then $H^{i-1}(G,\Z^k/p\Z^k) \ne 0$ and $H^i(G,\Z^k/p\Z^k) \ne 0$.
\item We have $H^i(G,\Z^k/p\Z^k) \ne 0$ if and only if $p \Md |H^i(G,\Z^k)|$ or $p \Md |H^{i+1}(G,\Z^k)|$.
\end{enumerate}
\end{Corollary}

Note that Lemma \ref{bockstein general} does not tell us anything about $H^0(G,\Z^k)$, because this group is not finite if it is not trivial. We will take this group into consideration in the special case of the Specht modules.

We choose $G := \Sig n$ and replace $k$ by $\rk(\Slz)$, $\Z^k$ by $\Slz$, and $\Z^k/p\Z^k$ by $\Slp$. For $i \ge 1$ let $x_i^\la$ be the $p$-rank of $\Hlz i$, and for $i \ge 0$ let $d_i^\la$ be the dimension of $\Hlp i$. By Lemma \ref{bockstein general}, we have $d_i^\la = x_i^\la + x_{i+1}^\la$ for $i \ge 1$, and for $i = 0$ we obtain the following:

\begin{Lemma} \label{bockstein H0}
\begin{enumerate}
\item $ \Hlz 0 \left\{\begin{array}{ll} \cong \Z & \mbox{for } \la = (n), \\ =0 & \mbox{otherwise.} \end{array}\right. $
\item For $\la = (n)$ we have $d_0^\la = 1$ and $x_1^\la = 0$, and for $\la \ne (n)$ we have $d_0^\la = x_1^\la$.
\end{enumerate}
\end{Lemma}
\begin{proof}
(a) Since $\Slq$ is irreducible (cf. \cite[2.1.11 Theorem]{JK}) and $\Slq \cong \Q \otimes_\Z \Slz$, the set of fixed points $\Hlz 0$ of $\Slz$ is $\Slz \cong \Z$ for $\la = (n)$ and trivial otherwise.

(b) Let $\la = (n)$. Obviously, the dimension of $\Hlp 0$ is $d_0^\la = 1$. Further, we have $\Hlz 0 \cong \Z$ by (a). In this case $\delta_0$ is surjective, that is, the Bockstein homomorphism $\beta_0$ maps to zero. Hence $\pi_1 : \Hlz 1 \rightarrow \Hlz 1$ is bijective. Since $\pi_1$ is multiplication by $p$, the group order of $\Hlz 1$ is not divisible by $p$.

If $\la \ne (n)$, then $\Hlz 0 = 0$ by (a), and the Bockstein homomorphism $\beta_0$ is injective. Since $\Fp^{x_1^\la} \cong \ker(\pi_1) = \im(\beta_0) \cong \Hlp 0 \cong \Fp^{\delta_0^\la}$, the claim follows.
\end{proof}

\begin{Corollary} \label{bockstein specht}
Let $i \ge 1$. By induction we obtain from Lemmas \ref{bockstein H0} and \ref{bockstein general} the following:
\begin{enumerate}
\item Let $\la \ne (n)$. Then $x_i^\la = \sum_{j=0}^{i-1} (-1)^{i-1-j}d_j^\la$.
\item Let $\la  =  (n)$. Then $x_i^\la = \sum_{j=1}^{i-1} (-1)^{i-1-j}d_j^\la$.
\end{enumerate}
\end{Corollary}

Now Lemma \ref{graph connection} is a direct consequence of Corollary \ref{bockstein cor} and Lemma \ref{bockstein H0}.
\\ 

As a last point in this section, we turn our attention to the relation between the $p$-block and the cohomology of a Specht module. Let $F$ be a field with $char(F) = p > 0$, and $M$ an $FG$-module. By \cite[Definition 6.1.2, Theorem 2.7.6]{We} and \cite[(1.2)]{Ev}, we have
$$
H^i(G,M) \cong {\rm Ext}^i_{\Z G}(\Z,M) \cong {\rm Ext}^i_{FG}(F,M) \cong H^i({\rm Hom}_{FG}({\mathbf P},M))
$$
for a projective resolution $\mathbf P$ of the trivial module $F$, consisting of projective modules $P_i, i \ge 0$. This resolution can be chosen such that every $P_i$ belongs to the principal block. If $M$ is an indecomposable $FG$-module that does not belong to the principal block, we have for every $i \ge 0$
$$
H^i(G,M) \cong H^i(Hom_{FG}({\mathbf P},M)) = Hom_{FG}(P_i,M) = 0.
$$

Further, the so-called {\em Nakayama Conjecture} (cf. \cite[6.1.21]{JK}) tells us that the $p$-blocks of $\Sig n$ are characterized by the $p$-cores of the partitions $\la \vdash n$: the modules $\Slp$ and $S_{\Fp}^\mu$ belong to the same block if and only if the $p$-cores of $\la$ and $\mu$ are the same.

The Specht module $\Slp$ belongs to the principal block, that is, the block containing $S_{\Fp}^{(n)}$, if and only if the $p$-core of $\la$ has at most one row. This leads to the following statement:

\begin{Corollary} \label{pcore coho}
If the $p$-core of $\la \vdash n$ has more than one row, we have $\Hlp i = 0$, and hence $p \nmid |\Hlz i|$ for all $i \ge 1$ by Corollary \ref{bockstein cor}.
\end{Corollary}

\section{Successors and predecessors} \label{branch}

This section provides the proof of Theorem \ref{Cp structure}. An important ingredient is the {\em Branching Theorem} for Specht modules. It describes the structure of restricted and induced Specht modules. The Specht module $\SlR$, restricted to $\Sigma_{n-1}$, is denoted by $\Slr$, and the Specht module $\SlR$, induced to $\Sigma_{n+1}$, is denoted by $\Sli$.

If $R$ is a field of characteristic $0$, we obtain a direct sum of Specht modules in both cases (cf. \cite[9.2 The Branching Theorem]{J}). Otherwise, we have only a {\em Specht filtration} instead of a direct sum. A Specht filtration of an $R\Sig n$-module $M$ is a series of pure submodules $0 =: M_0 \le M_1 \le \cdots \le M_q := M$, where $M_i/M_{i-1}$ is isomorphic to a Specht module $S^{\mu_i}$ for all $1 \le i \le q$. The series $S^{\mu_1}, \ldots, S^{\mu_q}$ is called a {\em Specht series}. This generalized version of the Branching Theorem holds for any integral domain $R$ (cf. \cite[9.3 Theorem, 17.14 Corollary]{J}.)
\\ 

What we want to prove now is the following: Let $p$ be a prime and $\la \vdash n \in \Nn$ with $p \Md |\Hlz i|$.
\begin{enumerate}
\item There exists $\mu \in \la+$ with $p \Md |H^i(\Sig {n+1}, \Smz)|$.
\item If $p \nmid n$, then there exists $\mu \in \la-$ with $p \Md |H^i(\Sig {n-1}, \Smz)|$.
\end{enumerate}
\begin{proof}
(a) By {\em Shapiro's Lemma} (cf. \cite[Lemmas 6.3.2, 6.3.4]{We}), we have $\Hlz{i} \cong H^i({\Sig n}_{+1},\Slzi)$. By the Branching Theorem, $\Slzi$ has a Specht series whose factors range over $\la+$. Thus, the prime divisor $p$ of $|\Hlz{i}| = |H^i(\Sig{n+1},\Slzi)|$ has to divide $|H^i(\Sig{n+1},\Smz)|$ for some $\mu \in \la+$.

(b) There is an $\alpha \in \Hlz{i}$ with order $p$. Since $p \nmid n$ we have
$$
0 \ne n\alpha = [\Sig n : \Sig {n-1}]\alpha = {\rm tr} \circ \res^{\Sig n}_{\Sig{n-1}} (\alpha)
$$
(cf. \cite[Transfer Maps 6.7.16, Lemma 6.7.17]{We}). Thus we have $\res^{\Sig n}_{\Sig{n-1}} (\alpha) \ne 0$. Hence, the order of $\res^{\Sig n}_{\Sig{n-1}} (\alpha)$ is $p$, and so $p \Md |H^i({\Sig n}_{-1},\Slzr)|$. By the Branching Theorem, $\Slzr$ has a Specht series whose factors range over $\la-$. Thus, the prime divisor $p$ has to divide $|H^i({\Sig n}_{-1},S^\mu_\Z)|$ for some $\mu \in \la-$.
\end{proof}

\section{Cohomology of Specht modules over $\Fp$} \label{comp}

In \cite[5. Some applications]{BKM}, low-degree cohomology of irreducible $\Fp\Sig n$-modules belonging to hook partitions and certain two-part partitions is considered. If the composition series of a Specht module consists only of such modules, we can determine the cohomology of this Specht module via long exact cohomology sequences in some cases. That is the aim of this section.

In the following let $p$ be a prime and $\la \vdash n \in \Nn$. If $\la$ is $p$-regular (that is, if $a_i < p$ for $\la = (\tilde{\la}_1^{a_1}, \tilde{\la}_2^{a_2}, \ldots, \tilde{\la}_m^{a_m})$) with $\tilde{\la}_1 > \ldots > \tilde{\la}_m$, then $\Slp$ has a unique irreducible quotient, denoted by $D^\la = \Slp/\rad(\Slp)$. These $D^\la$ form a complete set of representatives for the irreducible $\Fp\Sig n$-modules. Further, if $D$ is a composition factor of $rad(\Slp)$, then $D \cong D^\mu$ for some $\mu \rhd \la$, where $\lhd$ denotes the standard partial ({\em dominance}) order on partitions. And if $\la$ is not $p$-regular, all the composition factors of $\Slp$ have the form $D^\mu$ with $\mu \rhd \la$ (cf. \cite[7.1.14 Theorem]{JK}).
\\ 

For the case $p=2$ the reader is referred to \cite{KP}. In the following, as in \cite{BKM}, let $p$ be an odd prime.

First we take a look at hook partitions. In the case of odd primes we know their composition series (cf. \cite[24.1 Theorem]{J}): Let $\la := (n-j,1^j)$ for some $0 \le j \le n-1$.
\begin{enumerate}
\item If $p \nmid n$, then $\Slp$ is irreducible.
\item If $p \Md n$, there exist distinct irreducible modules $h_{n,k}$, $0 \le k \le n-2$, such that $\Slp$ has a composition series $0 \le M \le \Slp$ with $M \cong h_{n,j-1}$ and $\Slp/M \cong h_{n,j}$, where $h_{n,-1}$ and $h_{n,n-1}$ are interpreted as zero.
\end{enumerate}
Note that $h_{n,j} = D^\la$ if $\la$ is $p$-regular.

From \cite{BKM} we obtain the following statements:

\begin{Proposition} \label{H1hook} {\rm \cite[Proposition 5.2]{BKM}.}
We have $H^1(\Sig n, h_{n,j}) = 0$, except when
\begin{enumerate}
\item $p \mid n$, $j=1$;
\item $p = 3$, $j = 3$.
\end{enumerate}
In the exceptional cases $\dim H^1(\Sig n, h_{n,j}) = 1$.
\end{Proposition}

\begin{Proposition} \label{H2hook} {\rm \cite[Proposition 5.3]{BKM}.}
Assume that $j \ne 3$ if $p = 3 \mid n$. Then $H^2(\Sig n, h_{n,j}) = 0$, except when
\begin{enumerate}
\item $p \mid n$, $j = 2$;
\item $p = 3$, $j = 6$;
\item $p = 3 \mid n$, $j = 4$;
\item $p = 3 \nmid n$, $j = 3$;
\item $p = 3 = n$, $j=1$.
\end{enumerate}
In the exceptional cases $\dim H^2(\Sig n, h_{n,j}) = 1$.
\end{Proposition}

Now we have only to consider the reducible Specht modules. For the following three corollaries we suppose that $p \Md n$ and $\la := (n-j,1^j)$ for some $1 \le j \le n-2$, where $j \ne 2$ if $p = 3 = n$.

\begin{Corollary} \label{hook coho 0}
Since we can view zero cohomology as the set of fixed points, we have
$$
d^\la_0 = \left\{ \begin{array}{ll} 1 & \mbox{for } j = 1, \\ 0, & \mbox{otherwise.} \end{array} \right.
$$
\end{Corollary}

\begin{Corollary} \label{hook coho 1}
If $[p,j] \ne [3,3]$, then we have
$$
d^\la_1 = \left\{ \begin{array}{ll} 1 & \mbox{for } j \in \{1,2\} \mbox{ or } [p,j]=[3,4], \\ 0, & \mbox{otherwise.} \end{array} \right.
$$
For $[p,j] = [3,3]$ we have $d^\la_1 \le 1$.
\end{Corollary}
\begin{proof}
For $j=1$ we get by Proposition \ref{H2hook} the exact cohomology sequence
$$
0 = H^1(\Sig n,D^{(n)}) \lra \Hlp 1
$$
$$
\lra H^1(\Sig n,D^{(n-1,1)}) \lra H^2(\Sig n,D^{(n)}) = 0.
$$
In this case we have $\Hlp 1 \cong H^1(\Sig n,D^{(n-1,1)}) \cong \Fp$ by Proposition \ref{H1hook}.

For $j > 1$ and $[p,j] \ne [3,3]$ we get by Corollary \ref{hook coho 0} and Proposition \ref{H1hook} the exact cohomology sequence
$$
0 = H^0(\Sig n,h_{n,j}) \lra H^1(\Sig n,h_{n,j-1})
$$
$$
\lra \Hlp 1 \lra H^1(\Sig n,h_{n,j}) = 0.
$$
In this case we have by Proposition \ref{H1hook}
$$
\Hlp 1 \cong H^1(\Sig n,h_{n,j-1}) \left\{ \begin{array}{ll} \cong \Fp & \mbox{for } j = 2 \mbox{ or } [p,j]=[3,4], \\ = 0, & \mbox{otherwise.} \end{array} \right.
$$

Now let $p = j = 3$. In this case we have the exact cohomology sequence
$$
0 = H^1(\Sig n,D^{(n-2,1^2)}) \lra \Hlp 1 \lra H^1(\Sig n,h_{n,3}) \cong \F_3
$$
by Proposition \ref{H1hook}. Since we do not know enough about the continuation of the sequence, the only thing we can say is that $\dim(H^1(\Sig n,S^{(n-3,1^3)}_{\F_3})) \le 1$.
\end{proof}

\begin{Corollary} \label{hook coho 2}
\begin{enumerate}
\item For $p>3$ we have
$$
d^\la_2 \left\{ \begin{array}{ll}
  = 0, & \mbox{ if } j=1 \mbox{ or } j \ge 4, \\
  = 1, & \mbox{ if } j\in\{2,3\}.
\end{array} \right.
$$
\item For $p=3$ we have
$$
d^\la_2 \left\{ \begin{array}{ll}
        = 0, & \mbox{ if } j \ge 8 \mbox{ or } n>3,j=1, \\
      \le 1, & \mbox{ if } j=6 \mbox{ or } n=3,j=1, \\
        = 1, & \mbox{ if } j \in \{2,5,7\}. \\
\end{array} \right.
$$
\end{enumerate}
\end{Corollary}

\begin{proof} We consider the exact cohomology sequence
$$
H^1(\Sig n,h_{n,j}) \lra H^2(\Sig n,h_{n,j-1}) \lra \Hlp 2 \lra H^2(\Sig n,h_{n,j}).
$$
If we put in the results from Propositions \ref{H1hook} and \ref{H2hook}, we obtain the above statements by a similar argumentation as in Corollary \ref{hook coho 1}. For the cases $p>3,j=2$ and $p=3<n,j=2$ we additionally need
$$
d^\la_2 = x^\la_2 + x^\la_3 \ge x^\la_2 = d^\la_1 - d^\la_0 = 1
$$
by Lemma \ref{bockstein general} and Corollaries \ref{bockstein specht}, \ref{hook coho 0}, and \ref{hook coho 1}.
\end{proof}
Note that we do not have any statement if $p = 3$ and $j \in \{3,4\}$.
\\ 

Now we are going to take a closer look at partitions with two parts. For this we need the following relation on positive integers: For $a,b \in \N$ let the $p$-adic expansions be given by $a := \sum_{i=0}^r a_i p^i$ and $ b := \sum_{i=0}^s b_i p^i$ with $a_r \ne 0 \ne b_s$. Then we write $a \cp b$ if and only if $r < s$ and $a_i \in \{0,b_i\}$ for all $1 \le i \le r$.

Now let $\la := (n-m,m)$, $m \ge 1$. Every composition factor of $\Slp$ has the form $D^{(n-j,j)}$ for some $0 \le j \le m$. The multiplicity of $D^{(n-j,j)}$ as composition factor of $\Slp$ is 1 for $m-j \cp n - 2j + 1$ and 0 otherwise (cf. \cite[24.15 Theorem]{J}).

For reasons that become clear in Section \ref{path}, we are especially interested in composition series of the Specht modules $\Slp$ for $\la := (n-p,p)$. In a first step we determine the set of composition factors, and in a second step we determine their order.

For the following two corollaries let $n \ge 2p$, $\la := (n-p,p)$ and $n+1 \cmp j$ for some $0 \le j \le p-1$.

\begin{Corollary} \label{compfac 2p}
\begin{enumerate}
\item Let $1 \le k \le p-1$. Then $D^{(n-k,k)}$ is a composition factor of $\Slp$ if and only if $k = j$.
\item $D^{(n)}$ is a composition factor of $\Slp$ if and only if $n+1 \cmps {p+j}$.
\end{enumerate}
\end{Corollary}
\begin{proof}
(a) If $\frac{n+2-p}{2} \le k \le p-1$, then $D^{(n-k,k)}$ is no composition factor of $\Slp$, because then we have $1 \le n - 2k + 1 \le p-1$ and $1 \le p - k \le p-1$. This means that $p-k \not\cp n-2k+1$, because the highest exponents in their $p$-adic expansions are both 0. (Such a $k$ exists if and only if $n \le 3p - 4$.)

Now let $k < \frac{n+2-p}{2}$, which is equivalent to $n-2k+1 \ge p$. Let $\pad$ be the $p$-adic expansion of $n-2k+1$, where $c_r \ne 0$. Then we have $r \ge 1$; further, we have $p-k \cp n-2k+1$ if and only if $c_0 = p-k$. This is equivalent to $n-2k+1 \cmp {p-k}$. And this is equivalent to $n+1 \cmp k$, which means $k=j$. Note that
$$
\frac{n+2-p}{2} \ge \frac{2p+j+1-p}{2} > j.
$$
Now the statement follows.

(b) Let $n+1 := \pad$, where $c_r \ne 0$. Then we have $c_0 = j$, which means that $D^{(n)}$ is a composition factor of $\Slp$ if and only if $p-0 \cp n+1-0$. This is the case if and only if $r \ge 2$ and $c_1 = 1$. And this is equivalent to $n+1 \cmps {p+j}$, because $n+1 > 2p > p+j$.
\end{proof}

\begin{Corollary} \label{comp order}
The composition series of $\Slp$ has one of the following forms:
\begin{enumerate}
\item For $j = 0$ and $n+1 \not\cmps p$ we have $\Slp \cong D^\la$.
\item For $j = 0$ and $n+1 \cmps p$ we have a composition series $0 \le M \le \Slp$, where $M \cong D^{(n)}$ and $\Slp/M \cong D^\la$.
\item For $j > 0$ and $n+1 \not\cmps {p+j}$ we have a composition series $0 \le M \le \Slp$, where $M \cong D^{(n-j,j)}$ and $\Slp/M \cong D^\la$.
\item For $j > 0$ and $n+1 \cmps {p+j}$ we have a composition series $0 \le M_1 \le M_2 \le \Slp$, where $M_1 \cong D^{(n-j,j)}$, $M_2/M_1 \cong D^{(n)}$ and $\Slp/M_2 \cong D^\la$.
\end{enumerate}
\end{Corollary}
\begin{proof}
Since $\la$ is $p$-regular, the top factor is always $D^\la$. Hence the cases from (a) to (c) follow immediately from Corollary \ref{compfac 2p}. The missing part of (d) follows from \cite[24.4 Theorem]{J}, which says that for any partition $\mu \vdash n$, the Specht module $\Smp$ has an $\Fp\Sig n$-submodule isomorphic to the trivial module $D^{(n)}$ if and only if for all $1 \le i \le l(\mu)$ we have $\mu_i \equiv -1 (\mod p^{z_i})$, where $z_i := \min\{r \in \Nn \mid p^r > \mu_{i+1}\}$.

Note that with $\mu_{l(\mu)+1} = 0$ we have $z_{l(\mu)} = 0$. Thus we have the trivial condition $\mu_{l(\mu)} \equiv -1$ $(\mod 1)$.

In (d), $\Slp$ has no submodule isomorphic to $D^{(n)}$, because $\la_1 = n-p \cmps {j-1} \not\cmps {-1}$. Hence $D^{(n)}$ is the second factor.
\end{proof}

Note that the previous two corollaries also hold for $p=2$ if $n>4$.

Further, we need the following statements about cohomology:

\begin{Proposition} \label{H12part} {\rm \cite[($**$), page 176]{BKM}.}
Let $\pad$ be the $p$-adic expansion of $n+1$, where $c_r \ne 0$. Then we have
$$
H^1(\Sig n,D^{(n-j,j)}) \left\{ \begin{array}{ll} \cong \Fp, & \mbox{if } j = c_i p^i \mbox{ for some }i < r \mbox{, where } c_i > 0, \\ =0 & \mbox{otherwise.} \end{array} \right.
$$
\end{Proposition}

\begin{Proposition} \label{H22part1} {\rm \cite[Proposition 5.4]{BKM}.}
For $k<p$ we have $H^2(\Sig n,D^{(n-k,k)}) = 0$, except the cases
\begin{enumerate}
\item $p=3$, $n=3$, $k=1$;
\item $p=3$, $n=4$, $k=2$.
\end{enumerate}
In the exceptional cases $\dim H^2(\Sig n,D^{(n-k,k)}) = 1$ .
\end{Proposition}

\begin{Proposition} \label{H22part2} {\rm \cite[Proposition 5.5]{BKM}.}
$$
H^2(\Sig n,D^{(n-p,p)}) \left\{ \begin{array}{ll} = 0 & \mbox{ if } n+1 \equiv 0 (\mod p), \\ \cong \Fp & \mbox{ otherwise.} \end{array} \right.
$$
\end{Proposition}

If we put everything together, we obtain the following:

\begin{Corollary} \label{2part cohos}
Let $n \ge 2p$, $\la := (n-p,p)$ and $n+1 \cmp j$ for some $0 \le j \le p-1$.
\begin{enumerate}
\item If $j=0$, then $d^\la_2 = 0$ and $d^\la_0 = d^\la_1 = \left\{ \begin{array}{ll} 1 & \mbox{if } n+1 \cmps p, \\ 0, & \mbox{otherwise.} \end{array} \right.$
\item If $j>0$, then $d^\la_0 = 0$, $d^\la_1 = 1$ and $d^\la_2 \ge 1$.
\end{enumerate}
\end{Corollary}
\begin{proof}
(a) By Corollary \ref{comp order}, the compositions series of $\Slp$ has at most two factors, and they have the form $D^{(n)}$ or $D^{(n-p,p)}$. Because of Propositions \ref{H22part1} and \ref{H22part2}, we have $H^2(\Sig n,D^{(n)}) = 0 = H^2(\Sig n,D^{(n-p,p)})$. Now the first statement follows.

From \cite[24.4 Theorem]{J} (see above) we get the isomorphism type of $\Hlp 0$. Finally, the statement for $\Hlp 1$ follows from Lemmas \ref{bockstein general} and \ref{bockstein H0}.

(b) If $j>0$, then $\Hlp 0 = 0$ by \cite[24.4 Theorem]{J}.

For first- and second-degree cohomology we have to consider two cases by Corollary \ref{comp order}:

First case: Let $n+1 \not\cmps p+j$ as in Corollary \ref{comp order} (c). In the $p$-adic expansion of $n+1$ we have $c_0 = j$, $c_1 \ne 1$ and $r \ge 1$, because $n > p$. Hence we obtain the exact cohomology sequence
$$
0 = H^0(\Sig n,D^\la) \lra H^1(\Sig n,D^{(n-j,j)})
$$
$$
\lra H^1(\Sig n,\Slp) \lra H^1(\Sig n,D^\la) = 0.
$$
This means that $H^1(\Sig n,\Slp) \cong H^1(\Sig n,D^{(n-j,j)}) \cong \Fp$.

Second case: Let $n+1 \cmps {p+j}$ as in Corollary \ref{comp order} (d). In the $p$-adic expansion of $n+1$ we have $c_0 = j$, $c_1 = 1$ and $r \ge 1$, because $n > p$. Hence we obtain the exact cohomology sequence
$$
0 = H^0(\Sig n,M_2) \lra H^0(\Sig n,D^{(n)}) \cong \Fp \stackrel{\alpha}{\lra}
$$
$$
H^1(\Sig n,D^{(n-j,j)}) \cong \Fp \stackrel{\beta}{\lra} H^1(\Sig n,M_2) \lra H^1(\Sig n,D^{(n)}) = 0.
$$
Since $\alpha$ is injective, it is also bijective, and because of $\im(\alpha) = \ker(\beta)$, $\beta$ is the zero map. Since $\beta$ is also surjective, we have $H^1(\Sig n,M_2) = 0$.

Further, we get from Proposition \ref{H22part1} the exact cohomology sequence
$$
0 = H^2(\Sig n,D^{(n-j,j)}) \lra H^2(\Sig n,M_2) \lra H^2(\Sig n,D^{(n)}) = 0.
$$
(Note that $n \ne 3,4$ if $p=3$). Thus we have the exact cohomology sequence
$$
0 = H^1(\Sig n,M_2) \lra H^1(\Sig n,\Slp)
$$
$$
\lra H^1(\Sig n,D^\la) \lra H^2(\Sig n,M_2) = 0.
$$
Hence we get from Proposition \ref{H12part} that $H^1(\Sig n,\Slp) \cong H^1(\Sig n,D^\la) \cong \Fp$.

In both cases, $\Hlp 2 \ne 0$ follows from Lemma \ref{bockstein general} and Corollary \ref{bockstein specht} (a).
\end{proof}

\section{The Zassenhaus algorithm} \label{zas}

The basic ideas in this section were originally developed by Hans Zassenhaus in order to determine space groups (cf. \cite{Z}). The Zassenhaus algorithm can be translated into terms of cohomology (cf. \cite[4.2 Algorithmic determination]{HP}).

Before going into this, we take a look at a similar method for determining $H^1(G,M)$, where $G$ is a finite group and $M$ is an $RG$-lattice for a commutative ring $R$ with identity. We need a finite presentation of $G$ and a matrix representation of $G$ on $M$. Now, in \cite[7.6 Cohomology]{EHO} there is a description of how to build up certain matrices $Z$ and $B$ with the properties
$$
Z^1(G,M) \cong \ker(Z),
$$
$$
B^1(G,M) \cong \col_R(B),
$$
$$
H^1(G,M) \cong \ker(Z)/\col_R(B),
$$
where $\col_R(B)$ denotes the $R$-module generated by the columns of $B$. Note that we represent $M$ as column vectors here, in contrast to \cite{EHO}. Matrices of the type of $Z$ will be called {\em Zassenhaus matrices} in the following.

Now we come back to our special situation of $\Sig n$. For $n=0,1$ there is nothing to do. Suppose $n \ge 2$ in the following.

The most common presentation of $\Sig n$ is the Coxeter presentation, which is based on the neighbour transpositions. But we will choose it only for $n=2$. For $n \ge 3$ we will use the presentation
$$
G_n := \langle a,b \mid a^2, b^n, (ab)^{n-1}, (ab^jab^{n-j})^2 \mbox{ for all } 2 \leq j \leq \textstyle\frac{n}{2} \rangle
$$
(cf. \cite[6.2]{CM}), which is more suitable for our purposes, as we will see later on. The generator $a$ corresponds to the transposition $(1,2)$, and $b$ corresponds to the $n$-cycle $(1, \ldots, n)$. The number of relations is $r := \lfloor \frac{n}{2} \rfloor +2$.

The following statements are formulated only for the case $n \ge 3$. They hold in an analogous way for $n=2$.

We obtain a matrix representation of $\Sig n$ from the GAP library {\ttfamily spechtmats.g}. It provides matrices $A, B \in GL_k(\Z)$ corresponding to the generators $a, b$, where $k = \rk(\SlR)$. (In the following we identify $\SlR$ with $R^k$ together with the corresponding matrix operation.)

The Zassenhaus matrix for $R = \Z$ thus obtained is denoted by $Z_\la \in \Z^{rk \times 2k}$, and the matrix $B$ is denoted by $B_\la$ here. It has the form $B_\la = \left(\begin{array}{c} A-1 \\ B-1 \end{array}\right) \in \Z^{2k \times k}$.

In the case of $R = \Q$ we work with the same matrices. Since $H^1(\Sig n,\Q^k) = 0$ by \cite[Corollary 6.5.9]{We}, we have $\ker_\Q(Z_\la) = \col_\Q(B_\la)$ and hence the following:

\begin{Corollary} \label{ranks}
$\rk(Z_\la) + \rk(B_\la) = 2k$.
\end{Corollary}

But we can determine $H^1(\Sig n,\Slz)$ without knowing $\ker(Z_\la)$ explicitly: For $\la = (n)$ we have $H^1(\Sig n,\Slz) = 0$ by Lemma \ref{bockstein H0}. Now let $\la \ne (n)$. Because $S^\la_\Q$ is irreducible and because of \cite[Corollary 6.5.9]{We}, we have the exact cohomology sequence
$$
0 = H^0(\Sig n,\Q^k) \lra H^0(\Sig n,\Q^k/\Z^k) \lra H^1(\Sig n,\Z^k) \lra H^1(\Sig n,\Q^k) = 0.
$$
This means that
$$
\Hlz{1} \cong H^0(\Sig n,S^\la_\Q/\Slz) = \{ \bar{v} \in S^\la_\Q/\Slz \mid g\bar{v} = \bar{v} \mbox { for all } g \in \Sig n \}.
$$
Then $v + \Z^k \in H^0(\Sig n,\Q^k/\Z^k)$ if and only if $B_\la v \in \Z^{kl}$.

Now let the elementary divisors of $B_\la$ be given by $\hat{e}_1, \ldots, \hat{e}_q$. Then we have
$$
H^0(\Sig n,\Q^k/\Z^k) \cong \bigoplus_{i=1}^q\Z/\hat{e}_i\Z \oplus \bigoplus_{i=q+1}^k\Q/\Z.
$$
But $H^1(\Sig n,\Slz)$ is a finite abelian group by \cite[Corollary 6.5.10]{We}. Hence the following holds:

\begin{Theorem} \label{H^1(Sn)}
Let $\la \ne (n)$ and $\hat{e}_1, \ldots, \hat{e}_q$ be the elementary divisors of $B_\la$. Then we have $q = k$ and
$$
\Hlz{1} \cong \bigoplus_{i=1}^k\Z/\hat{e}_i\Z.
$$
\end{Theorem}

\begin{Corollary} \label{zas rk}
From Theorem \ref{H^1(Sn)} and Corollary \ref{ranks} we obtain
\begin{enumerate}
\item $B_{(n)} = 0$ and $\rk(B_\la) = k$ for $\la \ne (n)$,
\item $\rk(Z_{(n)}) = 2k$ and $\rk(Z_\la) = k$ for $\la \ne (n)$.
\end{enumerate}
\end{Corollary}

The actual Zassenhaus algorithm is concerned with something slightly different: Now we consider $\Hlz{2}$, which is isomorphic to $H^1(\Sig n,\Q^k/\Z^k)$. With a similar argument as above (cf. \cite[p. 128 ff]{Z}) we obtain:

\begin{Theorem} \label{H^2(Sn)}
Let $e_1,...,e_r$ be the elementary divisors of $Z_\la$. Then $\Hlz{2}$ is a finite abelian group with
$$
\Hlz{2} \cong \bigoplus_{i=1}^r\Z/e_i\Z.
$$
\end{Theorem}

Thus we compute the isomorphism type of $\Hlz 2$ by building up a Zassenhaus matrix and determining its elementary divisors. The problem is that the computation of elementary divisors of big matrices is a non-trivial problem. The size of a Zassenhaus matrix depends on the rank of the module and on the numbers of generators and relations in the chosen presentation. (This is the reason why we chose $G_n$.) And the ranks of Specht modules become very big very soon. Here is an example, just to give an impression of the dimensions we have to deal with: To determine that $(5,2^2,1^3) \in \C_3^2$, we have to evaluate a $29568 \times 7392$ Zassenhaus matrix.

The standard GAP function {\ttfamily ElementaryDivisorsMat} (cf. \cite{GAP}) is not strong enough. With the help of the function {\ttfamily ElementaryDivisorsPPartRk} from the GAP-package EDIM (cf. \cite{edim}), which uses a $p$-adic method, one achieves better results. But even then the end of the rope is reached quite soon. All in all, we have computed the elementary divisors for all $\la \vdash n \le 11$ and for those $\la \vdash n \le 20$ for which the corresponding Zassenhaus matrices are not too big. For $n = 20$ only 14 of the total 627 partitions are processed. A table with all computed results can be found in Appendix \ref{tables}.

But there are also a few examples we can do by hand:

\begin{Lemma} \label{some Z 1cohos}
\begin{enumerate}
\item For $n \ge 2$, $H^1(\Sig n,S^{(1^n)}_\Z)$ is cyclic of order $2$.
\item For $n \ge 2$, $H^1(\Sig n,S^{(n-1,1)}_\Z)$ is cyclic of order $n$.
\end{enumerate}
\end{Lemma}

\begin{Lemma} \label{some Z 2cohos}
\begin{enumerate}
\item For $n \ge 2$, $H^2(\Sig n,S^{(n)}_\Z)$ is cyclic of order $2$.
\item For $n \ge 2$, $H^2(\Sig n,S^{(1^n)}_\Z)$ is trivial except when $n \in \{3,4\}$. In the exceptional cases, $H^2(\Sig n,S^{(1^n)}_\Z)$ is cyclic of order $3$.
\item For $n \ge 3$, $H^2(\Sig n,S^{(n-1,1)}_\Z)$ is cyclic of order $2$ if $n$ is even, and trivial if $n$ is odd.
\end{enumerate}
\end{Lemma}

Both of the previous Lemmas can be proved by writing down the matrices $B_\la$ or $Z_\la$ respectively and determining their elementary divisors explicitly. For details the reader is referred to \cite{web}.

\section{Paths in Specht cohomology graphs} \label{path}

At the end of the introduction we formulated a plan to find paths in $\Cp 2$. This section is concerned with the first steps in that direction. In the following let $p > 2$.

\begin{Lemma} \label{(n-2,1^2) in Cp2}
Let $n \ge 3$ and $\la := (n-2,1^2)$. Then $\la \in \Cp{2}$ if and only if $p \Md n$. In this case, $x_2^\la = 1$. There is no predecessor of $\la$ in $\Cp 2$. If $n=p$, then $\la$ is the only partition of $n$ in $\Cp{2}$.
\end{Lemma}
\begin{proof}
By Proposition \ref{H2hook}, we have $\la \notin \C^2(\Fp) \supseteq \Cp 2$ for $p \nmid n$. Now let $p \Md n$. By Corollary \ref{hook coho 0}, we have $\Hlp 0 = 0$, and by Corollary \ref{hook coho 1}, we have $\Hlp 1 \cong \Fp$. Hence $x_2^\la = d_1^\la - d_0^\la = 1$, which means $\la \in \Cp 2$ by Corollary \ref{bockstein specht} (a).

The predecessors of $\la$ in $\Y$ are given by
$$
\la- = \{(n-3,1^2),(n-2,1)\}.
$$
Neither belongs to the principal block, and hence they are not in $\C^2(\Fp) \supseteq \Cp 2$ (cf. Corollary \ref{pcore coho}).

For $n=p$, only the hook partitions belong to the principal block. By Lemma \ref{some Z 2cohos}, $(n)$ and $(n-1,1)$ are not in $\Cp 2$ for all $n$, and $(1^n) \notin \Cp 2$ for $n > 3$. For $n=p>3$ and $3 \le j \le n-1$, we have $(n-j,1^j) \notin \C^1(\Fp) \supseteq \Cp 2$ by Corollary \ref{hook coho 1}. Hence the last statement follows.
\end{proof}

The previous lemma implies that every odd prime divisor of $|H^2(\Sig n,S_\Z^{(n-2,1^2)})|$ is contained in the same elementary divisor of $H^2(\Sig n,S_\Z^{(n-2,1^2)})$. But an even stronger statement might hold:

\begin{Conjecture}
Let $n \ge 4$. Then we have
$$
H^2(\Sig n,S^{(n-2,1^2)}_\Z) \cong \left\{\begin{array}{ll} \Z/2n\Z & \mbox{for n odd,} \\ \Z/\frac{n}{2}\Z & \mbox{for n even.} \end{array}\right.
$$
\end{Conjecture}

This was computed in GAP with the help of the Zassenhaus algorithm for $n \le 20$.
\\ 

Lemma \ref{(n-2,1^2) in Cp2} tells us that for $m \in \N$ there is at least one path with start vertex $(mp-2,1^2)$.

\begin{Lemma} \label{(n-2,1^2) path}
Let $p > 3$ and $m \in \N$.
\begin{enumerate}
\item Every path in $\Cp 2$ starting in $(mp-2,1^2)$ has initial part
$$
(mp-2,1^2) \lra (mp-2,2,1) \lra \cdots \lra (mp-2,p-2,1).
$$
\item For $m=1$, every path in $\Cp 2$ starting in $(p-2,1^2)$ has initial part as in (a), continued by
$$
((p-2)^2,1) \lra ((p-2)^2,1^2) \lra (p-1,p-2,1^2).
$$
\end{enumerate}
\end{Lemma}
\begin{proof}
By Lemma \ref{(n-2,1^2) in Cp2}, we have paths beginning at $(mp-2,1^2)$. For $0 \le j \le p-4$ we have
$$
(mp-2,1+j,1)+ = \left\{ \begin{array}{ll}
 \{ (mp-1,1^2), (mp-2,2,1), (mp-2,1^3) \}	& \mbox{for } j=0, \\
 \{ (mp-1,1+j,1), (mp-2,2+j,1), \\
	(mp-2,1+j,2), (mp-2,1+j,1^2) \}	& \mbox{for } j>0.
\end{array}\right.
$$
The $p$-cores of these partitions are given by
$$
\begin{array}{rcl}
(mp-1,1+j,1)	& \mapsto & (p-1,1+j,1),\\
(mp-2,2+j,1)	& \mapsto & (j+1),	\\
(mp-2,1+j,2)	& \mapsto & (j,1),	\\
(mp-2,1+j,1^2)	& \mapsto & (p-2,1+j,1^2). \end{array}
$$
This means, by Corollary \ref{pcore coho}, that $(mp-2,2+j,1)$ is the only successor of $(mp-2,1+j,1)$ in $\Cp 2$ for all $0 \le j \le p-4$, and statement (a) follows.

Statement (b) can be proved analogously.
\end{proof}

\begin{Lemma} \label{p3 hook path}
Let $p = 3$ and $m \in \N$. There is a path in $\C_3^2$ with initial part
$$
(3m-2,1^2) \lra (3m-2,1^3) \lra (3m-1,1^3)
$$
$$
\lra (3m-1,1^4) \lra (3m-1,2,1^3).
$$
\end{Lemma}
\begin{proof}
By Lemma \ref{(n-2,1^2) in Cp2}, a path begins at $(3m-2,1^2)$. Since $3 \nmid 3m+1$, it follows that $S^{(3m-2,1^3)}_{\F_3}$ is irreducible, but it is not isomorphic to the trivial module. From that and from Proposition \ref{H1hook} and Corollary \ref{bockstein specht} we obtain
$$
x^{(3m-2,1^3)}_2 = d^{(3m-2,1^3)}_1 - d^{(3m-2,1^3)}_0 = 1.
$$

Now the only successor of $(3m-2,1^3)$ in the principal block is $(3m-1,1^3)$.

Since $3 \mid 3m+3$, it follows that $S^{(3m-1,1^4)}_{\F_3}$ has no trivial module in its socle. From that and from Corollaries \ref{hook coho 1} and \ref{bockstein specht} we obtain
$$
x^{(3m-1,1^4)}_2 = d^{(3m-1,1^4)}_1 - d^{(3m-1,1^4)}_0 = 1.
$$

Finally, the only successor of $(3m-1,1^4)$ in the principal block is $(3m-\nolinebreak1,2,1^3)$.
\end{proof}

Note that in the first and third steps there might be more successors than the given ones.

\begin{Conjecture}
The only successor of $(3m-1,1^3)$ in $\C_3^2$ is $(3m-1,1^4)$.
\end{Conjecture}

For $m=1$, the only possible successor of $(3m-2,1^2)$ is $(3m-2,1^3)$. But in general there can be more than one successor:

\begin{Conjecture}
Let $n \ge 5$. Then we have
$$
H^2(\Sig n,S^{(n-3,2,1)}_\Z) \cong \left\{\begin{array}{ll} \Z/(n-1)\Z & \mbox{for } 3 \nmid n-1, \\ \Z/\frac{n-1}{3}\Z & \mbox{for } 3 \Md n-1. \end{array}\right.
$$
\end{Conjecture}

Again, this was computed in GAP for $n \le 20$. If the conjecture turns out to be true, this would mean that $(3m-2,2,1) \in \C^2_3$ if and only if $3 \Md m$.
\\ 

Now we are going to take a look at another type of path.

\begin{Lemma} \label{(n-p,p) in Cp2}
Let $n \ge 2p$ and $p \Md n$. Then $\la := (n-p,p) \in \Cp 2$ with $x_2^\la = 1$, and it has no predecessor in $\Cp 2$.
\end{Lemma}
\begin{proof}
By Corollaries \ref{2part cohos} (b) and \ref{bockstein specht} (a), we have $x_2^\la = 1$, which means $\la \in \Cp 2$. The predecessors of $\la$ in $\Y$ are $(n-p-1,p)$ for $n > 2p$ and $(n-p,p-1)$ for all $n$. But $(n-p,p-1)$ does not belong to the principal block, and $(n-p-1,p) \notin \Cp 2$ by Corollaries \ref{2part cohos} (a) and \ref{bockstein specht} (a).
\end{proof}

\begin{Lemma} \label{(n-p,p) path}
Let $m \in \N$. Every path in $\Cp 2$ starting in $(mp,p)$ has the initial part
$$
(mp,p) \lra (mp+1,p) \lra \cdots \lra (mp+p-2,p) \lra (mp+p-2,p,1).
$$
For the two-part partitions in this path we have $x_2^\la = 1$.
\end{Lemma}
\begin{proof}
The principle of the proof is the same as in Lemma \ref{(n-2,1^2) path}. We put in Lemma \ref{(n-p,p) in Cp2} and Corollaries \ref{2part cohos} and \ref{bockstein specht}.
\end{proof}

For odd $p$, we have seen two types of partitions in $\Cp 2$ without predecessor. At least for small $n$ those are the only ones, as computations in GAP show.

The graph $\C_2^2$ is much more involved than $\Cp 2$ for odd $p$. This is not very surprising, because of Lemma \ref{graph connection} and the fact that $\F_2$-cohomology is even less well understood than $\Fp$-cohomology for odd $p$. The tools we used above do not work in the case $p=2$, because too many partitions belong to the principal block.

But a few things can be observed:

\begin{Conjecture} \label{conj2}
Let $\la := (2l,2,1^q)$ for some $l \in \N$ and $q \in \Nn$. Then $\la \in \C_2^2$.
\end{Conjecture}

In Appendix \ref{graphs}, the $\C_2^2$ for $n \le 13$ can be found. By looking at it, one might conjecture that every partition $\la := (\la_1, \ldots, \la_m)$ with $\la_1, \ldots, \la_{m-1}$ odd and $\la_1 \ne 1$ is in $\C_2^2$. Unfortunately, this is not true. Two exceptions are computed, namely $(11,4)$ and $(11,5)$, where cohomology is trivial. There are probably more.

\begin{appendix}

\twocolumn[

\section{Tables} \label{tables}

In the following table we have $\lambda \vdash n$ and $k = \rk(S^\lambda_\mathbb{Z})$. In the column of "e.d." we list the nontrivial elementary divisors of $H^2(\Sigma_n,S^\lambda_\mathbb{Z})$ if existent. If $H^2(\Sigma_n,S^\lambda_\mathbb{Z})$ is trivial, that is, cyclic of order $1$, the column has the entry 1.

The elementary divisors were determined in GAP as the elementary divisors of the Zassenhaus matrix $Z_\lambda$ corresponding to the presentation $G_n$ and the representation given by the matrices from the library {\ttfamily spechtmats.g}. In some cases the elementary divisors could not be determined explicitly, but only their prime divisors. In these cases, the entries are set in brackets.

The partitions are ordered lexicographically for every $n$. If one or more partitions are skipped, this is marked by three dots.\\\\
]

\begin{center}

$
\begin{array}{|l|l||l|l|} \hline
n & \lambda & k &  e. d. \\ \hline\hline

2 & ( 1^2 )  & 1 & 1 \\ \hline
  & ( 2 )    & 1 & 2 \\ \hline\hline

3 & ( 1^3 )   & 1 & 3 \\ \hline
  & ( 2, 1 )  & 2 & 1 \\ \hline
  & ( 3 )     & 1 & 2 \\ \hline\hline

4 & ( 1^4 )     & 1 & 3 \\ \hline
  & ( 2, 1^2 )  & 3 & 2 \\ \hline
  & ( 2^2 )     & 2 & 2 \\ \hline
  & ( 3, 1 )    & 3 & 2 \\ \hline
  & ( 4 )       & 1 & 2 \\ \hline\hline

5 & ( 1^5 )     & 1 & 1  \\ \hline 
  & ( 2, 1^3 )  & 4 & 3  \\ \hline
  & ( 2^2, 1 )  & 5 & 4  \\ \hline
  & ( 3, 1^2 )  & 6 & 10 \\ \hline
  & ( 3, 2 )    & 5 & 2  \\ \hline
  & ( 4, 1 )    & 4 & 1  \\ \hline
  & ( 5 )       & 1 & 2  \\ \hline\hline

6 & ( 1^6 )       & 1  & 1  \\ \hline
  & ( 2, 1^4 )    & 5  & 3  \\ \hline
  & ( 2^2, 1^2 )  & 9  & 4  \\ \hline
  & ( 2^3 )       & 5  & 1  \\ \hline
\end{array}
$

$
\begin{array}{|l|l||l|l|} \hline
n & \lambda & k &  e. d. \\ \hline\hline

6 & ( 3, 1^3 )    & 10 & 2  \\ \hline
  & ( 3, 2, 1 )   & 16 & 5  \\ \hline
  & ( 3^2 )       & 5  & 6  \\ \hline
  & ( 4, 1^2 )    & 10 & 3  \\ \hline
  & ( 4, 2 )      & 9  & 2  \\ \hline
  & ( 5, 1 )      & 5  & 2  \\ \hline
  & ( 6 )         & 1  & 2  \\ \hline\hline

7 & ( 1^7 )       & 1  & 1  \\ \hline
  & ( 2, 1^5 )    & 6  & 1  \\ \hline
  & ( 2^2, 1^3 )  & 14 & 6  \\ \hline
  & ( 2^3, 1 )    & 14 & 1  \\ \hline
  & ( 3, 1^4 )    & 15 & 2  \\ \hline
  & ( 3, 2, 1^2 ) & 35 & 2  \\ \hline
  & ( 3, 2^2 )    & 21 & 1  \\ \hline
  & ( 3^2, 1 )    & 21 & 10 \\ \hline
  & ( 4, 1^3 )    & 20 & 3  \\ \hline
  & ( 4, 2, 1 )   & 35 & 2  \\ \hline
  & ( 4, 3 )      & 14 & 3  \\ \hline
  & ( 5, 1^2 )    & 15 & 14 \\ \hline
  & ( 5, 2 )      & 14 & 2  \\ \hline
  & ( 6, 1 )      & 6  & 1  \\ \hline
  & ( 7 )         & 1  & 2  \\ \hline\hline

\end{array}
$

$
\begin{array}{|l|l||l|l|} \hline
n & \lambda & k &  e. d. \\ \hline\hline

8 & ( 1^8 )        & 1  & 1  \\ \hline
  & ( 2, 1^6 )     & 7  & 1  \\ \hline
  & ( 2^2, 1^4 )   & 20 & 6  \\ \hline
  & ( 2^3, 1^2 )   & 28 & 1  \\ \hline
  & ( 2^4 )        & 14 & 1  \\ \hline
  & ( 3, 1^5 )     & 21 & 2  \\ \hline
  & ( 3, 2, 1^3 )  & 64 & 1  \\ \hline
  & ( 3, 2^2, 1 )  & 70 & 2  \\ \hline
  & ( 3^2, 1^2 )   & 56 & 2, 10  \\ \hline
  & ( 3^2, 2 )     & 42 & 2  \\ \hline
  & ( 4, 1^4 )     & 35 & 1  \\ \hline
  & ( 4, 2, 1^2 )  & 90 & 2  \\ \hline
  & ( 4, 2^2 )     & 56 & 2  \\ \hline
  & ( 4, 3, 1 )    & 70 & 3  \\ \hline
  & ( 4^2 )        & 14 & 2  \\ \hline
  & ( 5, 1^3 )     & 35 & 6  \\ \hline
  & ( 5, 2, 1 )    & 64 & 7  \\ \hline
  & ( 5, 3 )       & 28 & 2  \\ \hline
  & ( 6, 1^2 )     & 21 & 4  \\ \hline
  & ( 6, 2 )       & 20 & 2  \\ \hline
  & ( 7, 1 )       & 7  & 2  \\ \hline
  & ( 8 )          & 1  & 2  \\ \hline\hline

9 & ( 1^9 )          & 1   & 1  \\ \hline
  & ( 2, 1^7 )       & 8   & 1  \\ \hline
  & ( 2^2, 1^5 )     & 27  & 2  \\ \hline
  & ( 2^3, 1^3 )     & 48  & 3  \\ \hline
  & ( 2^4, 1 )       & 42  & 1  \\ \hline
  & ( 3, 1^6 )       & 28  & 2  \\ \hline
  & ( 3, 2, 1^4 )    & 105 & 1  \\ \hline
  & ( 3, 2^2, 1^2 )  & 162 & 2  \\ \hline
  & ( 3, 2^3 )       & 84  & 1  \\ \hline
  & ( 3^2, 1^3 )     & 120 & 2, 2  \\ \hline
  & ( 3^2, 2, 1 )    & 168 & 1  \\ \hline
  & ( 3^3 )          & 42  & 2  \\ \hline
  & ( 4, 1^5 )       & 56  & 1  \\ \hline
  & ( 4, 2, 1^3 )    & 189 & 2  \\ \hline
  & ( 4, 2^2, 1 )    & 216 & 1  \\ \hline
  & ( 4, 3, 1^2 )    & 216 & 5  \\ \hline
  & ( 4, 3, 2 )      & 168 & 3  \\ \hline
  & ( 4^2, 1 )       & 84  & 1  \\ \hline
  & ( 5, 1^4 )       & 70  & 6  \\ \hline
  & ( 5, 2, 1^2 )    & 189 & 1  \\ \hline
  & ( 5, 2^2 )       & 120 & 2  \\ \hline
  & ( 5, 3, 1 )      & 162 & 14 \\ \hline
  & ( 5, 4 )         & 42  & 2  \\ \hline
  & ( 6, 1^3 )       & 56  & 1  \\ \hline
  & ( 6, 2, 1 )      & 105 & 8  \\ \hline
  & ( 6, 3 )         & 48  & 3  \\ \hline
  & ( 7, 1^2 )       & 28  & 18 \\ \hline
  & ( 7, 2 )         & 27  & 2  \\ \hline
\end{array}
$

$
\begin{array}{|l|l||l|l|} \hline
n & \lambda & k &  e. d. \\ \hline\hline

9 & ( 8, 1 )         & 8   & 1  \\ \hline
  & ( 9 )            & 1   & 2  \\ \hline\hline

10 & ( 1^{10} )       & 1   & 1  \\ \hline
   & ( 2, 1^8 )       & 9   & 1  \\ \hline
   & ( 2^2, 1^6 )     & 35  & 2  \\ \hline
   & ( 2^3, 1^4 )     & 75  & 3  \\ \hline
   & ( 2^4, 1^2 )     & 90  & 1  \\ \hline
   & ( 2^5 )          & 42  & 1  \\ \hline
   & ( 3, 1^7 )       & 36  & 2  \\ \hline
   & ( 3, 2, 1^5 )    & 160 & 1  \\ \hline
   & ( 3, 2^2, 1^3 )  & 315 & 2  \\ \hline 
   & ( 3, 2^3, 1 )    & 288 & 1  \\ \hline 
   & ( 3^2, 1^4 )     & 225 & 2, 2  \\ \hline
   & ( 3^2, 2, 1^2 )  & 450 & 2  \\ \hline
   & ( 3^2, 2^2 )     & 252 & 1  \\ \hline
   & ( 3^3, 1 )       & 210 & 2  \\ \hline
   & ( 4, 1^6 )       & 84  & 1  \\ \hline
   & ( 4, 2, 1^4 )    & 350 & 2  \\ \hline
   & ( 4, 2^2, 1^2 )  & 567 & 1  \\ \hline 
   & ( 4, 2^3 )       & 300 & 1  \\ \hline
   & ( 4, 3, 1^3 )    & 525 & 1  \\ \hline
   & ( 4, 3, 2, 1 )   & 768 & 15 \\ \hline 
   & ( 4, 3^2 )       & 210 & 3  \\ \hline
   & ( 4^2, 1^2 )     & 300 & 1  \\ \hline
   & ( 4^2, 2 )       & 252 & 1  \\ \hline
   & ( 5, 1^5 )       & 126 & 2  \\ \hline
   & ( 5, 2, 1^3 )    & 448 & 3  \\ \hline
   & ( 5, 2^2, 1 )    & 525 & 1  \\ \hline
   & ( 5, 3, 1^2 )    & 567 & 2  \\ \hline 
   & ( 5, 3, 2 )      & 450 & 2, 2  \\ \hline
   & ( 5, 4, 1 )      & 288 & 7  \\ \hline
   & ( 5^2 )          & 42  & 10 \\ \hline
   & ( 6, 1^4 )       & 126 & 1  \\ \hline
   & ( 6, 2, 1^2 )    & 350 & 4  \\ \hline
   & ( 6, 2^2 )       & 225 & 2  \\ \hline
   & ( 6, 3, 1 )      & 315 & 2  \\ \hline
   & ( 6, 4 )         & 90  & 2  \\ \hline
   & ( 7, 1^3 )       & 84  & 6  \\ \hline
   & ( 7, 2, 1 )      & 160 & 3  \\ \hline
   & ( 7, 3 )         & 75  & 6  \\ \hline
   & ( 8, 1^2 )       & 36  & 5  \\ \hline
   & ( 8, 2 )         & 35  & 2  \\ \hline
   & ( 9, 1 )         & 9   & 2  \\ \hline
   & ( 10 )           & 1   & 2  \\ \hline\hline

11 & ( 1^{11} )       & 1    & 1    \\ \hline
   & ( 2, 1^9 )       & 10   & 1    \\ \hline
   & ( 2^2, 1^7 )     & 44   & 2    \\ \hline
   & ( 2^3, 1^5 )     & 110  & 1    \\ \hline
   & ( 2^4, 1^3 )     & 165  & 3    \\ \hline
   & ( 2^5, 1 )       & 132  & 1    \\ \hline
\end{array}
$

$
\begin{array}{|l|l||l|l|} \hline
n & \lambda & k &  e. d. \\ \hline\hline

11 & ( 3, 1^8 )       & 45   & 2    \\ \hline
   & ( 3, 2, 1^6 )    & 231  & 1    \\ \hline
   & ( 3, 2^2, 1^4 )  & 550  & 2    \\ \hline 
   & ( 3, 2^3, 1^2 )  & 693  & 1    \\ \hline 
   & ( 3, 2^4 )       & 330  & 1    \\ \hline
   & ( 3^2, 1^5 )     & 385  & 2,2  \\ \hline 
   & ( 3^2, 2, 1^3 )  & 990  & 1    \\ \hline 
   & ( 3^2, 2^2, 1 )  & 990  & 2    \\ \hline 
   & ( 3^3, 1^2 )     & 660  & 2,2  \\ \hline 
   & ( 3^3, 2 )       & 462  & 2    \\ \hline 
   & ( 4, 1^7 )       & 120  & 1    \\ \hline
   & ( 4, 2, 1^5 )    & 594  & 2    \\ \hline 
   & ( 4, 2^2, 1^3 )  & 1232 & 1    \\ \hline 
   & ( 4, 2^3, 1 )    & 1155 & 1    \\ \hline 
   & ( 4, 3, 1^4 )    & 1100 & 1    \\ \hline 
   & ( 4, 3, 2, 1^2 ) & 2310 & 1    \\ \hline 
   & ( 4, 3, 2^2 )    & 1320 & 3    \\ \hline 
   & ( 4, 3^2, 1 )    & 1188 & 5    \\ \hline 
   & ( 4^2, 1^3 )     & 825  & 1    \\ \hline 
   & ( 4^2, 2, 1 )    & 1320 & 1    \\ \hline 
   & ( 4^2, 3 )       & 462  & 1    \\ \hline 
   & ( 5, 1^6 )       & 210  & 2    \\ \hline
   & ( 5, 2, 1^4 )    & 924  & 3    \\ \hline 
   & ( 5, 2^2, 1^2 )  & 1540 & 1    \\ \hline 
   & ( 5, 2^3 )       & 825  & 1    \\ \hline 
   & ( 5, 3, 1^3 )    & 1540 & 2    \\ \hline 
   & ( 5, 3, 2, 1 )   & 2310 & 1    \\ \hline 
   & ( 5, 3^2 )       & 660  & 2,6  \\ \hline 
   & ( 5, 4, 1^2 )    & 1155 & 1    \\ \hline 
   & ( 5, 4, 2 )      & 990  & 2    \\ \hline 
   & ( 5^2, 1 )       & 330  & 14   \\ \hline
   & ( 6, 1^5 )       & 252  & 1    \\ \hline
   & ( 6, 2, 1^3 )    & 924  & 2    \\ \hline 
   & ( 6, 2^2, 1 )    & 1100 & 1    \\ \hline 
   & ( 6, 3, 1^2 )    & 1232 & 1    \\ \hline 
   & ( 6, 3, 2 )      & 990  & 1    \\ \hline 
   & ( 6, 4, 1 )      & 693  & 4    \\ \hline 
   & ( 6, 5 )         & 132  & 5    \\ \hline
   & ( 7, 1^4 )       & 210  & 2    \\ \hline
   & ( 7, 2, 1^2 )    & 594  & 2    \\ \hline 
   & ( 7, 2^2 )       & 385  & 2    \\ \hline
   & ( 7, 3, 1 )      & 550  & 2,18 \\ \hline 
   & ( 7, 4 )         & 165  & 2    \\ \hline
   & ( 8, 1^3 )       & 120  & 3    \\ \hline
   & ( 8, 2, 1 )      & 231  & 10   \\ \hline
   & ( 8, 3 )         & 110  & 1    \\ \hline
   & ( 9, 1^2 )       & 45   & 22   \\ \hline
   & ( 9, 2 )         & 44   & 2    \\ \hline
   & ( 10, 1 )        & 10   & 1    \\ \hline
   & ( 11 )           & 1    & 2    \\ \hline\hline

\end{array}
$

$
\begin{array}{|l|l||l|l|} \hline
n & \lambda & k &  e. d. \\ \hline\hline

12 & ( 1^{12} )      & 1    & 1    \\ \hline
   & ( 2, 1^{10} )   & 11   & 1    \\ \hline
   & ( 2^2, 1^8 )    & 54   & 2    \\ \hline
   & ( 2^3, 1^6 )    & 154  & 1    \\ \hline
   & ( 2^4, 1^4 )    & 275  & 3    \\ \hline 
   & ( 2^5, 1^2 )    & 297  & 1    \\ \hline 
   & ( 2^6 )         & 132  & 1    \\ \hline 
   & ( 3, 1^9 )      & 55   & 2    \\ \hline 
   & ( 3, 2, 1^7 )   & 320  & 1    \\ \hline 
   & ( 3, 2^2, 1^5 ) & 891  & 2    \\ \hline 
   & ( 3, 2^3, 1^3 ) & 1408 & 1    \\ \hline 
   & ( 3, 2^4, 1 )   & 1155 & 1    \\ \hline 
   & ( 3^2, 1^6 )    & 616  & 2,2  \\ \hline 
   & ( 3^2, 2, 1^4 ) & 1925 & 1    \\ \hline 
   & ( 3^2, 2^2, 1^2 ) & 2673 & 2  \\ \hline 
   & ( 3^2, 2^3 )    & 1320 & 1    \\ \hline 
   & ( 3^3, 1^3 )    & 1650 & 2,2  \\ \hline 
   & ( 3^3, 2, 1 )   & 2112 & 1    \\ \hline 
   & ( 3^4 )         & 462  & 2    \\ \hline 
   & ( 4, 1^8 )      & 165  & 1    \\ \hline 
   & ( 4, 2, 1^6 )   & 945  & 2    \\ \hline 
   & ( 4, 2^2, 1^4 ) & 2376 & 1    \\ \hline 
   & ( 4, 2^3, 1^2 ) & 3080 & 1    \\ \hline 
   & ( 4, 2^4 )      & 1485 & 1    \\ \hline 
   & ( 4, 3, 1^5 )   & 2079 & 1    \\ \hline 
   & ( 4, 3, 2, 1^3 )& 5632 & \qquad ?    \\ \hline
   & ( 4, 3, 2^2, 1 )& 5775 & \qquad ?    \\ \hline
   & ( 4, 3^2, 1^2 ) & 4158 & \qquad ?    \\ \hline
   & ( 4, 3^2, 2 )   & 2970 & 1    \\ \hline 
   & ( 4^2, 1^4 )    & 1925 & 1    \\ \hline 
   & ( 4^2, 2, 1^2 ) & 4455 & \qquad ?    \\ \hline
   & ( 4^2, 2^2 )    & 2640 & 1    \\ \hline 
   & ( 4^2, 3, 1 )   & 2970 & 1    \\ \hline 
   & ( 4^3 )         & 462  & 1    \\ \hline 
   & ( 5, 1^7 )      & 330  & 2    \\ \hline 
   & ( 5, 2, 1^5 )   & 1728 & 1    \\ \hline 
   & ( 5, 2^2, 1^3 ) & 3696 & (3)  \\ \hline 
   & ( 5, 2^3, 1 )   & 3520 & 1    \\ \hline 
   & ( 5, 3, 1^4 )   & 3564 & (2)  \\ \hline 
   & ( 5, 3, 2, 1^2 )& 7700 & \qquad ?    \\ \hline
   & ( 5, 3, 2^2 )   & 4455 & \qquad ?    \\ \hline
   & ( 5, 3^2, 1 )   & 4185 & \qquad ?    \\ \hline
   & ( 5, 4, 1^3 )   & 3520 & 1    \\ \hline 
   & ( 5, 4, 2, 1 )  & 5775 & \qquad ?    \\ \hline
   & ( 5, 4, 3 )     & 2112 & 3    \\ \hline 
   & ( 5^2, 1^2 )    & 1485 & 14   \\ \hline 
   & ( 5^2, 2 )      & 1320 & 2    \\ \hline 
   & ( 6, 1^6 )      & 462  & 1    \\ \hline 
   & ( 6, 2, 1^4 )   & 2100 & 2    \\ \hline 
   & ( 6, 2^2, 1^2 ) & 3564 & 1    \\ \hline 
\end{array}
$

$
\begin{array}{|l|l||l|l|} \hline
n & \lambda & k &  e. d. \\ \hline\hline

12 & ( 6, 2^3 )      & 1925 & 1    \\ \hline 
   & ( 6, 3, 1^3 )   & 3696 & 1    \\ \hline 
   & ( 6, 3, 2, 1 )  & 5632 & \qquad ?    \\ \hline
   & ( 6, 3^2 )      & 1650 & 1    \\ \hline 
   & ( 6, 4, 1^2 )   & 3080 & 2    \\ \hline 
   & ( 6, 4, 2 )     & 2673 & 4    \\ \hline 
   & ( 6, 5, 1 )     & 1155 & 2    \\ \hline 
   & ( 6^2 )         & 132  & 1    \\ \hline
   & ( 7, 1^5 )      & 462  & 2    \\ \hline
   & ( 7, 2, 1^3 )   & 1728 & 1    \\ \hline 
   & ( 7, 2^2, 1 )   & 2079 & 4    \\ \hline 
   & ( 7, 3, 1^2 )   & 2376 & 2,2  \\ \hline 
   & ( 7, 3, 2 )     & 1925 & 6    \\ \hline 
   & ( 7, 4, 1 )     & 1408 & 3    \\ \hline 
   & ( 7, 5 )        & 297  & 10   \\ \hline
   & ( 8, 1^4 )      & 330  & 3    \\ \hline
   & ( 8, 2, 1^2 )   & 945  & 2    \\ \hline 
   & ( 8, 2^2 )      & 616  & 1    \\ \hline 
   & ( 8, 3, 1 )     & 891  & 5    \\ \hline 
   & ( 8, 4 )        & 275  & 2    \\ \hline
   & ( 9, 1^3 )      & 165  & 2    \\ \hline
   & ( 9, 2, 1 )     & 320  & 11   \\ \hline
   & ( 9, 3 )        & 154  & 6    \\ \hline
   & ( 10, 1^2 )     & 55   & 6    \\ \hline
   & ( 10, 2 )       & 54   & 2    \\ \hline
   & ( 11, 1 )       & 11   & 2    \\ \hline
   & ( 12 )          & 1    & 2    \\ \hline\hline

13 & ( 1^{13} )       & 1    & 1    \\ \hline
   & ( 2, 1^{11} )    & 12   & 1    \\ \hline
   & ( 2^2, 1^9 )     & 65   & 2    \\ \hline
   & ( 2^3, 1^7 )     & 208  & 1    \\ \hline 
   & ( 2^4, 1^5 )     & 429  & 1    \\ \hline 
   & ( 2^5, 1^3 )     & 572  & 3    \\ \hline 
   & ( 2^6, 1 )       & 429  & 1    \\ \hline 
   & ( 3, 1^{10} )    & 66   & 2    \\ \hline 
   & ( 3, 2, 1^8 )    & 429  & 1    \\ \hline 
   & ( 3, 2^2, 1^6 )  & 1365 & 2    \\ \hline 
   & ( 3, 2^3, 1^4 )  & 2574 & 1    \\ \hline 
   & ( 3, 2^4, 1^2 )  & 2860 & 1    \\ \hline 
   & ( 3, 2^5 )       & 1287 & 1    \\ \hline 
   & ( 3^2, 1^7 )     & 936  & 2,2  \\ \hline 
   & ( 3^2, 2, 1^5 )  & 3432 & 1    \\ \hline 
   & ... & & \\ \hline

   & ( 3^3, 1^4 )     & 3575 & (2)  \\ \hline 
   & ... & & \\ \hline

   & ( 3^3, 2^2 )     & 3432 & 1    \\ \hline 
   & ... & & \\ \hline

   & ( 3^4, 1 )       & 2574 & 2    \\ \hline 
   & ( 4, 1^9 )       & 220  & 1    \\ \hline 
   & ( 4, 2, 1^7 )    & 1430 & 2    \\ \hline 
   & ... & & \\ \hline

\end{array}
$

$
\begin{array}{|l|l||l|l|} \hline
n & \lambda & k &  e. d. \\ \hline\hline

13 & ( 4, 3, 1^6 )    & 3640 & 1    \\ \hline 
   & ... & & \\ \hline

   & ( 4, 3^3 )       & 3432 & 1    \\ \hline 
   & ... & & \\ \hline

   & ( 4^3, 1 )       & 3432 & 1    \\ \hline 
   & ( 5, 1^8 )       & 495  & 2    \\ \hline 
   & ( 5, 2, 1^6 )    & 3003 & 1    \\ \hline 
   & ... & & \\ \hline

   & ( 5, 4^2 )       & 2574 & 1    \\ \hline 
   & ... & & \\ \hline

   & ( 5^2, 3 )       & 3432 & (2,3)\\ \hline 
   & ( 6, 1^7 )       & 792  & 1    \\ \hline 
   & ... & & \\ \hline

   & ( 6^2, 1 )       & 1287 & 2    \\ \hline 
   & ( 7, 1^6 )       & 924  & 2    \\ \hline 
   & ... & & \\ \hline

   & ( 7, 3^2 )       & 3575 & (2,3)\\ \hline 
   & ... & & \\ \hline

   & ( 7, 5, 1 )      & 2860 & 2,6  \\ \hline 
   & ( 7, 6 )         & 429  & 2    \\ \hline
   & ( 8, 1^5 )       & 792  & 1    \\ \hline 
   & ( 8, 2, 1^3 )    & 3003 & 6    \\ \hline 
   & ( 8, 2^2, 1 )    & 3640 & 1    \\ \hline 
   & ... & & \\ \hline

   & ( 8, 3, 2 )      & 3432 & 1    \\ \hline 
   & ... & & \\ \hline

   & ( 8, 5 )         & 572  & 5    \\ \hline 
   & ( 9, 1^4 )       & 495  & 2    \\ \hline 
   & ( 9, 2, 1^2 )    & 1430 & 1    \\ \hline 
   & ( 9, 2^2 )       & 936  & 1    \\ \hline 
   & ( 9, 3, 1 )      & 1365 & 22   \\ \hline 
   & ( 9, 4 )         & 429  & 2,2  \\ \hline
   & ( 10, 1^3 )      & 220  & 3    \\ \hline
   & ( 10, 2, 1 )     & 429  & 4    \\ \hline
   & ( 10, 3 )        & 208  & 3    \\ \hline
   & ( 11, 1^2 )      & 66   & 26   \\ \hline
   & ( 11, 2 )        & 65   & 2    \\ \hline
   & ( 12, 1 )        & 12   & 1    \\ \hline
   & ( 13 )           & 1    & 2    \\ \hline\hline

14 & ( 1^{14} )      & 1    & 1    \\ \hline
   & ( 2, 1^{12} )   & 13   & 1    \\ \hline
   & ( 2^2, 1^{10} ) & 77   & 2    \\ \hline
   & ( 2^3, 1^8 )    & 273  & 1    \\ \hline 
   & ( 2^4, 1^6 )    & 637  & 1    \\ \hline 
   & ( 2^5, 1^4 )    & 1001 & 3    \\ \hline 
   & ( 2^6, 1^2 )    & 1001 & 1    \\ \hline 
   & ( 2^7 )         & 429  & 1    \\ \hline 
   & ( 3, 1^{11} )   & 78   & 2    \\ \hline
   & ( 3, 2, 1^9 )   & 560  & 1    \\ \hline 
   & ( 3, 2^2, 1^7 ) & 2002 & 2    \\ \hline 
   & ... & & \\ \hline

\end{array}
$

$
\begin{array}{|l|l||l|l|} \hline
n & \lambda & k &  e. d. \\ \hline\hline

14 & ( 3^2, 1^8 )    & 1365 & 2,2  \\ \hline 
   & ... & & \\ \hline

   & ( 4, 1^{10} )   & 286  & 1    \\ \hline 
   & ( 4, 2, 1^8 )   & 2079 & 2    \\ \hline 
   & ... & & \\ \hline

   & ( 5, 1^9 )      & 715  & 2    \\ \hline 
   & ... & & \\ \hline

   & ( 6, 1^8 )      & 1287 & 1    \\ \hline 
   & ... & & \\ \hline

   & ( 7, 1^7 )      & 1716 & 2    \\ \hline 
   & ... & & \\ \hline

   & ( 7^2 )         & 429  & 14   \\ \hline 
   & ( 8, 1^6 )      & 1716 & 1    \\ \hline 
   & ... & & \\ \hline

   & ( 8, 6 )        & 1001 & 1    \\ \hline 
   & ( 9, 1^5 )      & 1287 & 2    \\ \hline 
   & ... & & \\ \hline

   & ( 9, 5 )        & 1001 & 2,2  \\ \hline 
   & ( 10, 1^4 )     & 715  & 1    \\ \hline 
   & ( 10, 2, 1^2 )  & 2079 & 4    \\ \hline 
   & ( 10, 2^2 )     & 1365 & 1    \\ \hline 
   & ( 10, 3, 1 )    & 2002 & 3    \\ \hline 
   & ( 10, 4 )       & 637  & 2    \\ \hline 
   & ( 11, 1^3 )     & 286  & 6    \\ \hline 
   & ( 11, 2, 1 )    & 560  & 13   \\ \hline 
   & ( 11, 3 )       & 273  & 2    \\ \hline 
   & ( 12, 1^2 )     & 78   & 7    \\ \hline
   & ( 12, 2 )       & 77   & 2    \\ \hline
   & ( 13, 1 )       & 13   & 2    \\ \hline
   & ( 14 )          & 1    & 2    \\ \hline\hline

15 & ( 1^{15} )      & 1    & 1    \\ \hline
   & ( 2, 1^{13} )   & 14   & 1    \\ \hline
   & ( 2^2, 1^{11} ) & 90   & 2    \\ \hline
   & ( 2^3, 1^9 )    & 350  & 1    \\ \hline 
   & ( 2^4, 1^7 )    & 910  & 1    \\ \hline 
   & ( 2^5, 1^5 )    & 1638 & 1    \\ \hline 
   & ( 2^6, 1^3 )    & 2002 & 3    \\ \hline 
   & ( 2^7, 1 )      & 1430 & 1    \\ \hline 
   & ( 3, 1^{12} )   & 91   & 2    \\ \hline 
   & ( 3, 2, 1^{10}) & 715  & 1    \\ \hline 
   & ( 3, 2^2, 1^8 ) & 2835 & 2    \\ \hline 
   & ... & & \\ \hline

   & ( 3^2, 1^9 )    & 1925 & 2,2  \\ \hline 
   & ... & & \\ \hline

   & ( 4, 1^{11} )   & 364  & 1    \\ \hline 
   & ( 4, 2, 1^9 )   & 2925 & 2    \\ \hline 
   & ... & & \\ \hline

   & ( 5, 1^{10} )   & 1001 & 2    \\ \hline 
   & ... & & \\ \hline

   & ( 6, 1^9 )      & 2002 & 1    \\ \hline 
   & ... & & \\ \hline

\end{array}
$

$
\begin{array}{|l|l||l|l|} \hline
n & \lambda & k &  e. d. \\ \hline\hline

15 & ( 7, 1^8 )      & 3003 & 2    \\ \hline 
   & ... & & \\ \hline

   & ( 8, 1^7 )      & 3432 & 1    \\ \hline 
   & ... & & \\ \hline

   & ( 8, 7 )        & 1430 & 7    \\ \hline 
   & ( 9, 1^6 )      & 3003 & 2    \\ \hline 
   & ... & & \\ \hline

   & ( 9, 6 )        & 2002 & 2    \\ \hline 
   & ( 10, 1^5 )     & 2002 & 1    \\ \hline 
   & ... & & \\ \hline

   & ( 10, 5 )       & 1638 & 5    \\ \hline 
   & ( 11, 1^4 )     & 1001 & 6    \\ \hline 
   & ( 11, 2, 1^2 )  & 2925 & 2    \\ \hline 
   & ( 11, 2^2 )     & 1925 & 1    \\ \hline 
   & ( 11, 3, 1 )    & 2835 & 26   \\ \hline 
   & ( 11, 4 )       & 910  & 1    \\ \hline 
   & ( 12, 1^3 )     & 364  & 1    \\ \hline 
   & ( 12, 2, 1 )    & 715  & 14   \\ \hline 
   & ( 12, 3 )       & 350  & 3    \\ \hline 
   & ( 13, 1^2 )     & 91   & 30   \\ \hline
   & ( 13, 2 )       & 90   & 2    \\ \hline
   & ( 14, 1 )       & 14   & 1    \\ \hline
   & ( 15 )          & 1    & 2    \\ \hline\hline

16 & ( 1^{16} )       & 1    & 1    \\ \hline
   & ( 2, 1^{14} )    & 15   & 1    \\ \hline
   & ( 2^2, 1^{12} )  & 104  & 2    \\ \hline
   & ( 2^3, 1^{10} )  & 440  & 1    \\ \hline 
   & ( 2^4, 1^8 )     & 1260 & 1    \\ \hline 
   & ( 2^5, 1^6 )     & 2548 & 1    \\ \hline 
   & ... & & \\ \hline

   & ( 2^7, 1^2 )     & 3432 & 1    \\ \hline 
   & ( 2^8 )          & 1430 & 1    \\ \hline 
   & ( 3, 1^{13} )    & 105  & 2    \\ \hline 
   & ( 3, 2, 1^{11} ) & 896  & 1    \\ \hline 
   & ... & & \\ \hline

   & ( 3^2, 1^{10} )  & 2640 & 2,2  \\ \hline 
   & ... & & \\ \hline

   & ( 4, 1^{12} )    & 455  & 1    \\ \hline 
   & ... & & \\ \hline

   & ( 5, 1^{11} )    & 1365 & 2    \\ \hline 
   & ... & & \\ \hline

   & ( 6, 1^{10} )    & 3003 & 1    \\ \hline 
   & ... & & \\ \hline

   & ( 8^2 )          & 1430 & 2    \\ \hline 
   & ... & & \\ \hline

   & ( 9, 7 )         & 3432 & (2,7)\\ \hline 
   & ... & & \\ \hline

   & ( 11, 1^5 )      & 3003 & 2    \\ \hline 
   & ... & & \\ \hline

   & ( 11, 5 )        & 2548 & 1    \\ \hline 
\end{array}
$

$
\begin{array}{|l|l||l|l|} \hline
n & \lambda & k &  e. d. \\ \hline\hline

16 & ( 12, 1^4 )      & 1365 & 1    \\ \hline 
   & ... & & \\ \hline

   & ( 12, 2^2 )      & 2640 & 2    \\ \hline 
   & ... & & \\ \hline

   & ( 12, 4 )        & 1260 & 2    \\ \hline 
   & ( 13, 1^3 )      & 455  & 6    \\ \hline 
   & ( 13, 2, 1 )     & 896  & 5    \\ \hline 
   & ( 13, 3 )        & 440  & 6    \\ \hline 
   & ( 14, 1^2 )      & 105  & 8    \\ \hline
   & ( 14, 2 )        & 104  & 2    \\ \hline
   & ( 15, 1 )        & 15   & 2    \\ \hline
   & ( 16 )           & 1    & 2    \\ \hline\hline

17 & ( 1^{17} )      & 1    & 1   \\ \hline
   & ( 2, 1^{15} )   & 16   & 1   \\ \hline
   & ( 2^2, 1^{13} ) & 119  & 2   \\ \hline
   & ( 2^3, 1^{11} ) & 544  & 1   \\ \hline 
   & ( 2^4, 1^9 )    & 1700 & 1   \\ \hline 
   & ... & & \\ \hline

   & ( 3, 1^{14} )   & 120  & 2   \\ \hline 
   & ( 3, 2, 1^{12} )& 1105 & 1   \\ \hline 
   & ... & & \\ \hline

   & ( 4, 1^{13} )   & 560  & 1   \\ \hline 
   & ... & & \\ \hline

   & ( 5, 1^{12} )   & 1820 & 2   \\ \hline 
   & ... & & \\ \hline

   & ( 13, 1^4 )     & 1820 & 2   \\ \hline 
   & ... & & \\ \hline

   & ( 13, 4 )       & 1700 & 2   \\ \hline 
   & ( 14, 1^3 )     & 560  & 3   \\ \hline 
   & ( 14, 2, 1 )    & 1105 & 16  \\ \hline 
   & ( 14, 3 )       & 544  & 1   \\ \hline 
   & ( 15, 1^2 )     & 120  & 34  \\ \hline
   & ( 15, 2 )       & 119  & 2   \\ \hline
   & ( 16, 1 )       & 16   & 1   \\ \hline
   & ( 17 )          & 1    & 2   \\ \hline\hline

18 & ( 1^{18} )      & 1    & 1  \\ \hline
   & ( 2, 1^{16} )   & 17   & 1  \\ \hline
   & ( 2^2, 1^{14} ) & 135  & 2  \\ \hline
   & ( 2^3, 1^{12} ) & 663  & 1  \\ \hline 
   & ( 2^4, 1^{10} ) & 2244 & 1  \\ \hline 
   & ... & & \\ \hline

   & ( 3, 1^{15} )   & 136  & 2  \\ \hline 
   & ( 3, 2, 1^{13} )& 1344 & 1  \\ \hline 
   & ... & & \\ \hline

   & ( 4, 1^{14} )   & 680  & 1  \\ \hline 
   & ... & & \\ \hline

   & ( 5, 1^13 )     & 2380 & 2  \\ \hline 
   & ... & & \\ \hline

\end{array}
$

$
\begin{array}{|l|l||l|l|} \hline
n & \lambda & k &  e. d. \\ \hline\hline

18 & ( 14, 1^4 )     & 2380 & 3  \\ \hline 
   & ... & & \\ \hline

   & ( 14, 4 )       & 2244 & 2  \\ \hline 
   & ( 15, 1^3 )     & 680  & 2  \\ \hline 
   & ( 15, 2, 1 )    & 1344 & 17 \\ \hline 
   & ( 15, 3 )       & 663  & 6  \\ \hline 
   & ( 16, 1^2 )     & 136  & 9  \\ \hline
   & ( 16, 2 )       & 135  & 2  \\ \hline
   & ( 17, 1 )       & 17   & 2  \\ \hline
   & ( 18 )          & 1    & 2  \\ \hline\hline

19 & ( 1^{19} )      & 1    & 1   \\ \hline
   & ( 2, 1^{17} )   & 18   & 1   \\ \hline
   & ( 2^2, 1^{15} ) & 152  & 2   \\ \hline
   & ( 2^3, 1^{13} ) & 798  & 1   \\ \hline 
   & ( 2^4, 1^{11} ) & 2907 & 1   \\ \hline 
   & ... & & \\ \hline

   & ( 3, 1^{16} )   & 153  & 2   \\ \hline 
   & ( 3, 2, 1^{14} )& 1615 & 1   \\ \hline 
   & ... & & \\ \hline

   & ( 4, 1^{15} )   & 816  & 1   \\ \hline 
   & ... & & \\ \hline

   & ( 15, 4 )       & 2907 & 2   \\ \hline 
   & ( 16, 1^3 )     & 816  & 3   \\ \hline 
   & ( 16, 2, 1 )    & 1615 & 6   \\ \hline 
   & ( 16, 3 )       & 798  & 3   \\ \hline 
   & ( 17, 1^2 )     & 153  & 38  \\ \hline
   & ( 17, 2 )       & 152  & 2   \\ \hline
   & ( 18, 1 )       & 18   & 1   \\ \hline
   & ( 19 )          & 1    & 2   \\ \hline\hline

20 & ( 1^{20} )      & 1    & 1  \\ \hline 
   & ( 2, 1^{18} )   & 19   & 1  \\ \hline 
   & ( 2^2, 1^{16} ) & 170  & 2  \\ \hline 
   & ( 2^3, 1^{14} ) & 950  & 1  \\ \hline 
   & ... & & \\ \hline

   & ( 3, 1^{17} )   & 171  & 2  \\ \hline 
   & ( 3, 2, 1^{15} )& 1920 & 1  \\ \hline 
   & ... & & \\ \hline

   & ( 4, 1^{16} )   & 969  & 1  \\ \hline 
   & ... & & \\ \hline

   & ( 17, 1^3 )     & 969  & 6   \\ \hline 
   & ( 17, 2, 1 )    & 1920 & 19  \\ \hline 
   & ( 17, 3 )       & 950  & 2   \\ \hline 
   & ( 18, 1^2 )     & 171  & 10  \\ \hline
   & ( 18, 2 )       & 170  & 2   \\ \hline
   & ( 19, 1 )       & 19   & 2   \\ \hline
   & ( 20 )          & 1    & 2   \\ \hline\hline
\end{array}
$

\end{center}

\onecolumn

\section{Graphs of type $\mathcal{C}_p^2$} \label{graphs}

On the basis of the previous tables, we order partitions in graphs of type $\mathcal{C}_p^2$. If for some vertex no successor is computed, this is marked by a question mark. If successors are known but not listed, this is marked by three dots. The edges are not marked by arrows, because the direction is clear.\\

$\mathbf{p = 2:}$\\
For purposes of clarity, $\mathcal{C}_2^2$ for $n \le 13$ is split up into two partial graphs. The first one contains the partitions related to Conjecture \ref{conj2} and those of type $(\lambda_1, \ldots, \lambda_m)$ with $\lambda_1, \ldots, \lambda_{m-1}$ odd and $\lambda_1 \ne 1$. The second one contains the rest, together with some connections to the first. All in all, every known edge of $\mathcal{C}_2^2$ for $n \le 13$ can be found in at least one of the two partial graphs.\\

\input{graph2.1.eepicemu}

\newpage

\input{graph2.2.eepicemu}

\newpage

$\mathbf{p = 3:}$\\

\input{graph3.eepicemu}

\newpage

\input{graph5.eepicemu}

\input{graph7.eepicemu}

\end{appendix}

\bibliographystyle{apalike}
\bibliography{specht.coho}

\textbf{Address of the author:}

Christian Weber, Lehrstuhl D f\"ur Mathematik, RWTH Aachen University, Templergraben 64, 52062 Aachen, Germany

email: christian.weber@math.rwth-aachen.de

\end{document}